\documentclass{article}

\usepackage[english]{babel}
\usepackage[utf8]{inputenc}
\usepackage{times}
\usepackage[T1]{fontenc}

\usepackage{amsmath}
\usepackage{amssymb}
\usepackage{amsthm}
\usepackage{mathrsfs}
\usepackage[all]{xy}
\usepackage{enumerate}

\theoremstyle{plain}
\newtheorem{thm}{Theorem}[section]
\newtheorem{prop}[thm]{Proposition}

\newtheorem{conj}[thm]{Conjecture}

\theoremstyle{definition}
\newtheorem{dfn}{Definition}[section]

\theoremstyle{remark}

\newcommand{\bbP}{{\mathbb P}}
\newcommand{\G}{{\mathbb G}}

\newcommand{\cH}{{\mathcal H}}
\newcommand{\calH}{{\mathcal H}}

\newcommand{\cM}{{\mathcal M}}
\newcommand{\calM}{{\mathcal M}}
\newcommand{\cO}{{\mathcal O}}
\newcommand{\Z}{{\mathbb Z}}
\newcommand{\NN}{{\mathbb N}}
\newcommand{\bbN}{{\mathbb N}}
\newcommand{\fkX}{{\mathfrak X}}
\newcommand{\fkY}{{\mathfrak Y}}

\newcommand{\Q}{{\mathbb Q}}

\newcommand{\R}{{\mathbb R}}

\newcommand{\bbC}{{\mathbb C}}

\DeclareMathOperator{\et}{\acute{e}t}
\DeclareMathOperator{\proet}{pro-\acute{e}t}
\DeclareMathOperator{\fet}{f-\acute{e}t}

\DeclareMathOperator{\rk}{rk}

\DeclareMathOperator{\Ker}{Ker}

\DeclareMathOperator{\Stab}{Stab}

\DeclareMathOperator{\an}{an}

\DeclareMathOperator{\Gm}{\mbf G_m}

\DeclareMathOperator{\RNS}{RNS}
\DeclareMathOperator{\HTvLoc}{HTvLoc}
\DeclareMathOperator{\HTv}{HTv}

\DeclareMathOperator{\HT}{HT}

\DeclareMathOperator{\Gal}{Gal}

\DeclareMathOperator{\spec}{sp}

\DeclareMathOperator{\gf}{\pi_1}

\DeclareMathOperator{\get}{\pi_1^{\acute{e}t}}
\DeclareMathOperator{\ga}{\pi_1^{alg}}

\DeclareMathOperator{\gtemp}{\pi_1^{temp}}

\DeclareMathOperator{\ggeom}{\pi_1^{log-geom}}

\DeclareMathOperator{\Hom}{Hom}

\newcommand{\findem}{\end{proof}}
\newcommand{\dem}{\begin{proof}}

\newcommand{\da}{\begin{displaystyle}}
\newcommand{\db}{\end{displaystyle}}

\newcommand{\mbf}{\mathbf}

\newcommand{\mbb}{\mathbb}

\DeclareMathOperator{\C}{C}

\newcommand{\N}{\operatorname{N}}

\newcommand{\Isom}{\operatorname{Isom}}

\begin{document}

\title{Resolution of Non-Singularities and the Absolute Anabelian Conjecture}
\author{Emmanuel Lepage}
\maketitle

\section*{Introduction}

Grothendieck anabelian conjecture is concerned with the recovery, for smooth curves over number fields, of a curve from its étale fundamental group. More precisely, it conjectures that if $X$ and $Y$ are two hyperbolic curves over a number field $K$ with absolute Galois group $G_K$, then $\Isom_K(X,Y)\to \Isom_{G_K}(\ga(X),\ga(Y))/\sim$ is  bijective, where $\sim$ is given by composing with inner automorphisms. This conjecture was proved by S. Mochizuki in \cite{LocProP}, where it is proved much more generally for sub-$p$-adic fields.

In this article we will only be interested in the case where $K$ is a $p$-adic field. In \cite{AAG2}, Mochizuki then asks if one can remove the assumption of lying over an isomorphism of absolute Galois groups. This gives us the following conjecture:
\begin{conj}
Let $K,L$ be finite extensions of $\Q_p$, and $X/L$, $Y/K$ be two hyperbolic curves. Then the map $\Isom(X,Y)\to \Isom(\ga(X),\ga(Y))/\sim$ is  bijective.
\end{conj}
The conjecture was proved by Mochizuki in \cite{AAG2} in the case where $X$ and $Y$ are of quasi-Belyi type. This quasi-Belyi type condition doesn't contain any proper curve. The goal of this paper is to extend this result to a wider family of curves, that in particular contains every Mumford curve : 
\begin{thm}
Let $K,L$ be finite extensions of $\Q_p$, and $X/L$, $Y/K$ be two hyperbolic curves satisfying resolution of non singularities (see definition \ref{defRNS}), for example two hyperbolic Mumford curves. Then the map $\Isom(X,Y)\to \Isom(\ga(X),\ga(Y))/\sim$ is  bijective.
\end{thm}
A theorem of Mochizuki (\cite[Thm 2.9]{AAG2}) claims that to prove that it is sufficient to characterize group-theoretically among the partial Galois sections $G_{K'}\to\ga(X)$ those which come from a closed point $x\in X(K')$ for any finite extension $K'$ of $K$.
The idea of the proof then relies on the reconstruction of the Berkovich space $\lvert X\rvert$ (as a topological space) of the curve from the étale fundamental group, as was done in \cite[Thm. 3.10]{RNSMumf} for the geometric tempered fundamental group. From a result of Mochizuki (\cite[Thm 2.7]{HypCurve}), one can recover from $\ga(X)$ the dual graph $\G(X)$ of the special fiber of the stable model of $X_{\overline K}$. By applying this reconstruction to all the finite étale coverings, the recovering of the Berkovich space follows from a homeomorphism $\lvert \overline{X}\rvert = (\varprojlim_{Z}\lvert \G(Z)\rvert)/\ga(X)$, where $Z$ goes through connected pointed finite étale covers of $X$ and $\overline{X}$ is the smooth compactification of $X$. It is then sufficient to characterized rigid points inside the Berkovich space $\lvert X\rvert$.

The Berkovich points on curves are traditionally classified into four classes, where the type 1 points correspond to $\bbC_p$-points. The type 2 points are type 3 points are easy to distinguish from the others by the local topological behavior. The main issue is then to distinguish between type 4 and type 1 points. In this paper, we present two different ways of doing it. First through a metric characterization, using Mochizuki's recovery of the log-structure of the special fiber of the stable model \cite[Thm 2.7]{HypCurve}). Secondly, by showing that the geometric decomposition group of type 4 points is non trivial. More precisely, we will see that a $\Z_p$-torsor of $X_{\bbC_p}$ does not decompose at a type 4 points if it is not a topological torsor. Locally in a small disk, this $\Z_p$-torsor comes from Kummer classes of invertible functions, and by applying the log-differential, whose image is bounded, one gets that these invertible functions cannot have $p^n$th root for big enough $n$. Then the rigid points are characterized among type 1 points by the fact that their image in the absolute Galois group $G_K$ is open.

\section{Resolution of Non-Singularities}

Let $K$ be a complete nonarchimedean field of mixed characteristics $(0,p)$. Let $O_K$ be its ring of integers and  $\widetilde K$ be its residue field.

Let $X=\overline X\backslash D$ be a hyperbolic curve over  $K$. Let $X^{\an}$ be the associated Berkovich space. As a set, $X^{\an}$ can be seen as the union of the set of closed points of $X$ and the set of $\R$-valued valuations $v$ on the fraction field $K(X)$ extending the valuation on $K$. The completion of $K(X)$ for this valuation will be denoted by $\calH(v)$. Since $K(X)/K$ is an extension of transcendance degree one, for every such valuation $v$, $r(v)+t(v)\leq 1$ where $r(v)=\rk(v(K(X)^\times)/v(K^\times))$ and $t(v)$ is the transcendance degree of the residue field $\widetilde{\calH(v)}$ of $\calH(v)$ as an extension of $\widetilde K$.

A type 2 point of $X^{\an}$ then corresponds to a valuation $v$ such that $t(v)=1$, and a type 3 point corresponds to a valuation $v$ such that $r(v)=1$. If a point $x$ of $X^{\an}$ is a rigid point or comes from a valuation such that $K(X)$ embeds isometrically into the completion of an algebraic closure of $K$, then $x$ is said to be of type $1$. Otherwise it is of type $4$.

If $\overline X$ has a stable model $\fkX$ over $O_K$, there is a natural specialization map $\spec:X^{\an}\to \fkX_s$ where $\fkX_s=\fkX_{\widetilde K}$ is the special fiber of $\fkX$.
Then the generic point $z$ of any irreducible component of $\fkX_s$ has a unique preimage $v_z$ and this preimage is a type 2 point, and the residue field $\widetilde{\calH(v_z)}$ is isomorphic to $k(z)$. Such a type 2 point will be called a skeletal point of $X^{\an}$, and the set of skeletal points will be denoted by $V(X)$. Skeletal points are the vertices of the minimal triangulation of $X^{\an}$ in the sense of \cite[\S 5.1.13]{ducroscurves}. The connected components of $X^{\an}\backslash V(X)$ are given by the preimages by $\spec{}$ of the closed points in $\fkX_s$. Moreover, if $K$ is algebraically closed and $x$ is a closed point of $\fkX_s$, then $\spec^{-1}(x)$ is analytically isomorphic to :
\begin{itemize}
\item a punctured open disk if $x$ is a cusp;
\item an open disk if $x$ is a smooth point;
\item an open annulus if $x$ is a node.
\end{itemize}
In particular if $v$ is a type two point of $X^{\an}$ such that $\calH(v)$ is the fraction field of a curve of genus $>0$, then $v$ is skeletal.

\begin{dfn} (\cite[Def. 2.1]{RNSMumf})\label{defRNS}

Let $\C$ be an algebraically closed complete nonarchimedean field of mixed characteristics $(0,p)$. Let $X$ be a hyperbolic curve over  $\C$.
$X$ is said to satisfy \emph{resolution of non-singularities} (or $\RNS(X)$ for short) if 
 for every point $x$ of type $2$ of $X^{\an}$, there exists a finite \'etale cover $f:Y\to X$ such that $f^{-1}(x)\cap V(Y)\neq \emptyset$ .

More generally, if $X$ is a hyperbolic curve over a complete field $K$ of mixed characteristics $(0,p)$, then  we will write $\RNS(X)$ for $\RNS(X_{\C})$, where $\C$ is the completion of an algebraic closure of $K$.
\end{dfn}

\begin{prop}
Let $f:X\to Y$ be a map of hyperbolic curves over $K$.
\begin{enumerate}
\item If $f:X\to Y$ is a dominant map, then $\RNS(Y)$ implies $\RNS(X)$.
\item If $f:X\to Y$ is finite étale, then $\RNS(X)$ and $\RNS(Y)$ are equivalent.
\end{enumerate}
\end{prop}
\begin{proof}
One can assume $K$ to be algebraically closed.
\begin{enumerate}
\item Let $x\in X^{\an}$ be a point of type 2, and let $y=f(x)$. Then $y$ is also of type $2$, and by $\RNS(Y)$, there exists a finite étale cover $g:Z\to Y$ and $z\in g^{-1}(y)\cap V(Z)$. Let $f_Z:X\times_YZ\to Z$ be the map obtained by pulling back $f$ and $g_X: X\times_YZ\to X$ be the finite étale cover obtained by pulling back $g$. Then there exists $z'\in g_X^{-1}(x)\cap f_Z^{-1}(z) \subset g_X^{-1}(x)\cap f_Z^{-1}(V(Z))\subset g_X^{-1}(x)\cap V(X\times_YZ)$.
\item By (i),  $\RNS(Y)$ implies $\RNS(X)$. Let us assume $\RNS(X)$ and let $y\in Y^{\an}$ be a point of type 2, and let $x\in X^{\an}$ be a preimage of $y$ by $f$. Then there exists a finite étale cover $g:Z\to X$ and $z\in g^{-1}(x)\cap V(Z)$. Then $fg:Z\to Y$ is a finite étale cover and $z\in (fg)^{-1}(y)\cap V(Z)$, so that $\RNS(Y)$ is true.
\end{enumerate}
\end{proof}

\begin{prop}(\cite[Th. 0.1]{RNSMumf})
\begin{enumerate}
\item If $\overline X$ is a Mumford curve over $\overline \Q_p$ and $X\subset \overline X_{\bbC_p}$ is a Zariski open subset of $X_{\bbC_p}$ which is a hyperbolic curve over $\bbC_p$, then $\RNS(X_{\bbC_p})$ is true.
\item If $X$ is a $\bbC_p$-curve of quasi-Belyi type, then $\RNS(X)$ is true.
\end{enumerate}
\end{prop}

The case of Mumford curves over $\overline{\Q}_p$ follows from a more general criterion in terms of the following Hodge-Tate map. 
Let $X$ be a smooth adic space over $K$, and let $\nu:X_{\proet}\to X_{\et}$ be the natural map from its pro-étale site to its étale site. Then the short exact sequence on $X_{\proet}$:
\[1\to \Z_p(1)\to \lim_{\times p}\cO^{\times}_{X}\to \cO^{\times}_X\to 1\]
induces a boundary map $\kappa:\cO_{X_{\et}}^\times=\nu_* \cO^{\times}_X\to R^1\nu_*\Z_p(1)$. There exists a unique $\cO_{X_{\et}}$-linear map $\phi:\Omega^1_{X_{\et}}\to R^1\nu_*\widehat \cO^\times_{X_{\et}}$ such that the following diagram commutes (see \cite[Lem. 3.24]{scholzeCDM}):
\[\xymatrix{\cO_{X_{\et}}^\times \ar[d]^{dlog} \ar[r]^\kappa &  R^1\nu_*\Z_p(1) \ar[d]^i\\
\Omega^1_{X_{\et}} \ar[r]^{\phi} &
R^1\nu_*\widehat \cO^\times_{X_{\et}}}
\]
and $\phi$ is an isomorphism.
By taking global sections of $\phi^{-1}i$, one gets a map:
\[\HT_X: H^1(X,\Z_p(1))\to H^0(X,\Omega_X^1).\]

\begin{dfn}
\begin{enumerate}
\item Let $\HTvLoc_1(X)\subset X(K)$ (for Hodge-Tate vanishing locus of X) be the set of points such that there exists$c\in H^1(Y,\Z_p(1))$ such that $HT_Y(c)$ has vanishing order $e<+\infty$ at $y$ such that $e+1$ is not a power of $p$.
\item Let $\HTvLoc^{\fet}_1(X)=\bigcup_{f:Y\to X}f(\HTvLoc_1(Y))\subset X(K)$, where the union is indexed through all finite étale covers of $X$.
\item Let $\HTv_1(X)$ be the following property: $\HTvLoc^{\fet}_1(X)$ is dense in $X(K)$.
\end{enumerate}
\end{dfn}

\begin{prop}
Let $X$ be a hyperbolic curve over $K$. Then $\HTv_1(X)$ implies $\RNS(X)$.
\end{prop}
\begin{proof}
By the argument of \cite[Prop. 2.2]{RNSMumf} it is enough to show that for every disk $D$ in $X$, there exists a finite étale cover $g:Z\to X$ such that $g^{-1}(D)\cap V(Z)\neq \emptyset$.

By definition of $\HTv_1(X)$, there exists a finite étale cover $f:Y\to X$, $y\in f^{-1}(D)$ and $c\in H^1(Y,\Z_p(1))$ such that the vanishing order $e$ of $HT_Y(c)$ at $y$ is not of the form $p^n-1$ for any $n\in\N$.

For $n\in\N$, let $c_n\in H^1(Y,\mu_{p^n})$ be the image of $c$ by the map induced from the surjection $\Z_p(1)\to\mu_{p^n}$ and let $g_n:Y_n\to Y$ be the corresponding $\mu_{p^n}$-torsor. Let $D'$ be a discal neighborhood of $y$ in $f^{-1}(D)$. Since $H^1(D',\Gm)=0$, the restriction $c_{n,D'}$ of $c_n$ to $D'$ is in the image of the Kummer map $\kappa_n:\cO^\times(D')\to H^1(D',\mu_{p^n})$. Let $a_n\in \cO^\times(D')$ such that $\kappa_n(a_n)=c_{n,D'}$ and $a_n(y)=1$. There exists $b_n\in \cO^\times(D')$ such that $a_n=b_n^{p^n}a_{n+1}$ and $b_n(y)=1$. After replacing $D'$ by a smaller disk $D''$, there exists $r<1$ such that for every $n\in\N$, $\lvert b_n-1\rvert_{D'}\leq r$, and therefore $(b_n^{p^n})_{n}$ converges uniformly to $1$ on $D'$. Therefore $(a_{n,D'})_n$ converges to $a\in \cO^\times(D'')$ and $c_{D''}$ is the image of $a$ by the boundary map $\kappa:\cO_{Y_{\et}}^\times=\nu_* \cO^{\times}_Y\to R^1\nu_*\Z_p(1)$. Therefore $HT_Y(c)_{D''}=dlog(a)$ and \cite[Prop. 2.4]{RNSMumf} shows that, for $n$ big enough, there exists $y_n\in V(Y_n)\cap g_n^{-1}(D'')\subset V(Y_n)\cap (fg_n)^{-1}(D)$. 
\end{proof}

\section{Reconstruction of the Berkovich topology}
If $X$ is a hyperbolic curve over $K$, let $\overline X$ be its smooth compactification, so that $X=\overline X\backslash D_X$ for some divisor $D_X$ on $\overline X$. Let $\lvert X\rvert$ be the Berkovich space of $X$, considered as a topological space.

If $\widetilde X$ is a universal pro-object in the category of finite étale cover of $X$, we will use the following notations : $\lvert \widetilde X\rvert=\varprojlim_{\widetilde X\to Y}\lvert Y\rvert$ and $\lvert \widetilde X\rvert_c=\varprojlim_{\widetilde X\to Y}\lvert \overline Y\rvert$, where $Y$ goes through $\widetilde X$-pointed finite étale cover of $X$.

\begin{prop}\label{BerkRecovery}
Let $X_1/K$ and $X_2/K_2$ be two p-adic hyperbolic curves satisfying rns and let $\widetilde X_i/\widetilde K_i$ be a universal pro-cover of $X_i/K_i$. Let $\alpha:\pi_1(X_1,\widetilde X_1)\to\pi_1(X_2,\widetilde X_2)$ be an isomorphism of fundamental groups. Then there is a unique homeomorphism of pro-Berkovich spaces $\tilde \alpha: \lvert \widetilde X_1\rvert \to \lvert \widetilde X_2\rvert$ and $\tilde\alpha_c : \lvert \widetilde X_1\rvert_c \to \lvert \widetilde X_2\rvert_c$ which are equivariant
in the sense that the following diagram commutes:
\[\xymatrix{\gf(X_1,\widetilde X_1)\times |\widetilde X_1|_c \ar[r]\ar[d]^{\alpha\times\tilde\alpha} & |\widetilde{X}_1|_c\ar[d]^{\tilde\alpha}\\
\gf(X_2,\widetilde X_2)\times |\widetilde X_2|_c \ar[r] & |\widetilde X_2|_c}
\]
In particular, by quotienting by the étale fundmental group and the geometric étale fundamental groups, one gets  homeomorphisms:
\[\bar\alpha:\overline X^{\an}_1\to\overline X_2^{\an},\]
\[\bar\alpha_{\bbC_p}:\overline X^{\an}_{1,\bbC_p}\to\overline X_{2,\bbC_p}^{\an},\]
\end{prop}
\begin{proof}
Let $\Pi_i=\pi_1(X_i,\widetilde X_i)$, $\Delta_i=\pi_1(X_{i,\overline K_i},\widetilde X_i)\subset\pi_1(X_i,\widetilde X_i)$ and $G_i=\Pi_i/\Delta_i\simeq \Gal(\overline K_i/K_i)$.
Then $\alpha(\Delta_1)=\Delta_2$ by \cite[Lem. 1.1.4]{HypCurve} and therefore $\alpha$ induces an isomorphism $\alpha_G:G_1\to G_2$.

Let $\Delta_i^{(p')}$ be the maximal pro-prime-to-$p$ quotient of $\Delta_i$ and $\Pi_i^{(p')}=\Pi_i/\Ker(\Delta_i\to\Delta_i^{(p')}) $ be the maximal geometrically pro-prime-to-$p$ quotient of $\Pi_i$. Let $\alpha^{(p')}_{geom}:\Delta_1^{(p')}\to \Delta_2^{(p')}$ and $\alpha^{(p')}:\Pi_1^{(p')}\to \Pi_2^{(p')}$ be the morphisms induced by $\alpha$.

Let $H_1$ be an open subgroup of $\Pi_1$, and let $H_2=\alpha(H_1)$. Let $Y_1$ be the connected  finite étale cover of $X_1$ corresponding to $H_1$ and let $Y_2$ be the cover of $X_2$ corresponding to $H_2$. Denote by $L_i$ the field of constants of $Y_i$. Then cofinally on $Y_1$, both $Y_1/L_1$ and $Y_2/L_2$ have stable reduction. Denote by $\mathfrak Y_1$ and $\mathfrak Y_2$ the corresponding stable models endowed with the natural log-structure. According to \cite[Thm 2.7]{HypCurve}, $\alpha$ induces natural isomorphisms of log-special fibers $\alpha_{H_1,\log}:\mathfrak Y_{1,s}\to \mathfrak Y_{2,s}$. More precisely, when one identifies the prime-to-$p$ geometric log fundamental group $\ggeom(\mathfrak Y_{1,s})$ of $\mathfrak Y_{i,s}$ with $(H_i\cap\Delta_i)^{(p')}$ along the specialisation map, the morphism $\ggeom(\mathfrak Y_{1,s})\to\ggeom(\mathfrak Y_{1,s})$ induced by $\alpha_{H_1,\log}$ coincides with $\alpha^{(p')}_{H_1,geom}:(H_1\cap\Delta_1)^{(p')}\to (H_2\cap\Delta_2)^{(p')}$. In particular, one gets an isomorphism of the geometric dual graphs $\G_{Y_1}\to \G_{Y_2}$ such that if $z_1$ is an edge or a vertex in $\G_{Y_1}$ with image $z_2$ in $\G_{Y_2}$, then $\alpha^{(p')}_{H_1,geom}(D_{z_1})=D_{z_2}$ where $D_{z_i}$ is the decomposition group of $z_i$ in $(H_i\cap\Delta_i)^{(p')}$. The construction (and unicity) of $\bar \alpha$ is then perfectly similar to the one in \cite[Thm. 3.10]{RNSMumf}, that we now recall.

If $K_i$ is an open subgroup of $\Delta_i$, let $\tilde V(K_i)$ (resp. $\tilde E(K_i)$) be the set of decomposition groups of vertices of $\G_{Y_i}$ in $K_i^{(p')}$ where $Y_i\to X_i$ is the finite étale cover corresponding to $H_i$, where $H_i$ is any open subgroup of $\Pi_i$ such that $K_i=H_i\cap\Delta_i$ and $Y_i$ has stable reduction ($\tilde V(K_i)$ and $\tilde E(K_i)$ do not depend on the choice of $H_i$). Let $V(K_i)=\tilde V(K_i)/K_i$ and $E(K_i)=\tilde E(K_i)/K_i$ be the set of conjugation classes.

Let $K''_i\subset K'_i\subset K_i$ be three open subsets of $\Delta_i$ such that $K'_i$ and $K''_i$ are normal in $K_i$. let $j_{K''_i,K'_i}:(K''_i)^{(p')}\to (K'_i)^{(p')}$ be the induced map. Then for every $D\in\tilde V(K'_i)$ there exists a unique $D'\in\tilde V(K''_i)$ up to conjugation in $K_i$ such that $j_{K''_i,K'_i}(D')$ is open in $D$. This induces an injection $i_{K''_i,K'_i}:V(K'_i)/K_i\to V(K''_i)/K_i$.

Let $D\in \tilde E(K_i)$. Let $K'$ be a normal open subgroup of $K_i$ and let $j:K'^{(p')}\to K^{(p')}$ be the induced map. One denotes by 
\[A_{D,K'}:=\{D'\in \tilde V(K')| j(D')\textrm{ is conjugate to an open subgroup of }D\}/K_i\subset V(K')/K_i.\]

 If $K''\subset K'$, then $i_{K'',K'}$ maps $A_{D,K''}$ to $A_{D,K'}$, and let 
 \[A_D=lim_{K'} A_{D,K'}.\]
 
Let $D_1\in\tilde V(K_i)$ such that $D\subset D_1$ and let $D_2$ be the only other subgroup in $\tilde V(K_i)$ such that $D\subset D_2$.
 The full subgraph $\mbb G(A_{D,K'})$ of $\mbb G_{K'}/K_i$ with vertices \[A_{D,K'}\cup \{i_{K',K_i}(D_1),i_{K',K_i}(D_2)\}\] is a line , so that $A_{D,K'}$ is naturally an ordered set for which $D_1$ is the minimal element and $D_2$ is the maximal. If one exchanges the role of $D_1$ and $D_2$, the order is reversed.

The injective map $A_{D,K''}\to A_{D,K'}$ is increasing, so that one obtains an order on $A_D$, which is then, non-canonically, isomorphic to $\mbf Q\cap (0,1)$ as an ordered set.
Let $\widehat A_D$ be the Dedekind completion of $A_D$: $\widehat A_D$ is an ordered topological space non-canonically isomorphic to $[0,1]$ as an ordered set. One denotes by $e(D_1)$ the minimal element of $A_D$. As a set $\widehat A_D$ does not depend on $D_1$ but the order is reversed when one replace $D_1$ by $D_2$. One then defines the topological space
\[S(K_i):=\left(V(K_i)\coprod_{D\in E(K_i)}\widehat A_D\right)/\sim\]
where $\sim$ is generated by \[\forall D\subset D_1,\quad D_1\sim e(D_1).\]
Since $\widehat A_D$ is (a priori non canonically) homeomorphic to $[0,1]$, $S(K_i)$ is (a priori non canonically) homeomorphic to the geometric realization of $\mbb G(K_i)$.

Let $\widetilde B_i=\varinjlim_{K_i}S(K_i)$. Then one has a natural homeomorphism $|\widetilde X_1|_c\to \widetilde B_i$ (see \cite[Lem. 3.9]{RNSMumf}). Then the family of isomorphisms of graphs $\mbb G(K_1)\to\mbb G(K_2)$ induce a homeomorphism $\widetilde B_1\to\widetilde B_2$, and therefore one gets a homeomorphism $\tilde \alpha_c:|\widetilde X_1|_c\to |\widetilde X_2|_c$. It preserves cusps, for example by \cite{GphAnab}, so that it also induces a homeomorphism $\tilde \alpha:|\widetilde X_1|\to |\widetilde X_2|$.  The action of $\Pi_i$ on itself by conjugacy induce an action of $\Pi_i$ on $\widetilde B_i$, so that the homeomorphism $\varinjlim_{K_i}S(K_i)\to |\widetilde X_1|_c$ is $\Pi_i$-equivariant, and $\widetilde B_1\to\widetilde B_2$ is also equivariant.

The uniqueness is identical to \cite[Prop. 3.11]{RNSMumf}.

The map $\bar\alpha_{\C_p}$ is deduced from the fact that $\alpha(\Delta_1)=\Delta_2$.

\end{proof}

\section{The canonical metric}
Let $\G$ be a finite connected graph. A metric on $\G$ is just a map $l:E(\G)\to\R_{>0}$ where $E(\G)$ is the set of edges of $\G$. If $x,y$ are two vertices of $\G$, a path from $x$ to $y$ is a sequence $p=(x=x_0,e_1,x_1,\cdots,e_n,x_n=y)$ where for all $i$, $x_i$ is a vertex of $\G$ and $e_i$ is an edge of $\G$ abutting at $x_{i-1}$ and $x_i$. We then set $d(x,y)=\inf_p l(p)$ where the infimum goes through all paths from $x$ to $y$ and $l(p)=\sum_{i=1}^n l(e_i)$ if $p=(x=x_0,e_1,x_1,\cdots,e_n,x_n=y)$ is a path.

Let $X^{\log}/k^{\log}$ be a semistable log curve over a log point, and let $\G_X$ be the dual graph of $X$. Let us fix $\pi\in M_{k^{\log}}\backslash\{0\}$. Let $e$ be a node of $X^{\log}$, then $\overline M_{X,e}$ is isomorphic to $\overline M_{k^{\log}} \oplus \NN u\oplus \NN v )/(u+v=a)$ for some $a\in \overline M_{k^{\log}}\backslash\{0\}$. Note that $u$ and $v$ are characterized as the generators of the faces of $\overline M_{X,e}$ isomorphic to $\N$, so that, in particular, $a$ is preserved by automorphisms of $\overline M_{X,e}$. Then $a=\pi^r$ for a unique $r\in\Q_{>0}$, and one can set $l_{X^{\log},\pi}(e)=r$ to get a metric $d_{X^{\log},\pi}$ on $\G_X$.

If $X$ is an analytic $\bbC_p$-curve in the sense of Berkovich, the set $X_{[2]}$ of points of type 2 is endowed with a canonical metric (see \cite[Prop. 4.5.7]{ducroscurves}). If $x,y\in X_{[2]}$, let an annular path from $x$ to $y$ be a sequence $p=(x=x_0,A_1,x_1,A_2,\cdots, A_n,x_n=y)$, where for all $i$, $x_i\in X_2$ and $A_i$ is a annulus embedded in $X$ with endpoints $x_{i-1}$ and $x_i$. The length of the path $p=(x=x_0,A_1,x_1,A_2,\cdots, A_n,x_n=y)$ is then defined to be $l(p)=\sum_{i=1}^n r_{\mathrm{norm}}(A_i)$ and the distance between $x$ and $y$ is $d(x,y)=\inf_pl(p)$ where the infimum goes through all annular paths from $x$ to $y$. If $X$ is topologically simply connected, $d(x,y)=l(p)$ for all annular path $p=(x=x_0,A_1,x_1,A_2,\cdots, A_n,x_n=y)$ such that, for all $i\neq j$, $A_i\cap A_j=\emptyset$.

Let $\fkX$ be a proper semistable log curve over $\bbC_p$ with smooth generic fiber. Let $\pi:\fkX_\eta^{\an}\to \fkX_s$ be the specialisation map from the analytic generic fiber to the special fiber. Let $V(\fkX_s)$ be the set of generic points of $\fkX_s$, $N(\fkX_s)$ be the set of nodal points of $\fkX_s$ and $S(\fkX)=\pi^{-1}(V(\fkX_s))$. Then $\pi$ induces a bijection $S(\fkX)\to V(\fkX_s)$ and $S(\fkX)$ is a triangulation of $\fkX_\eta^{\an}$ in the sense of \cite[\S 5.1.13]{ducroscurves}. If $z\in N(\fkX_s)$, then $\pi^{-1}(z)$ is an open annulus and $r_{\mathrm{norm}}(\pi^{-1}(z))=l_{\fkX_s^{\log},p}(z)$. Moreover, the distance between two points of $S(\fkX)$ is reached by an annular path  $p=(x=x_0,A_1,x_1,A_2,\cdots, A_n,x_n=y)$ such that for all $i$ there exists $z_i\in N(\fkX_s)$ such that $A_i=\pi^{-1}(z_i)$. When one identifies $S(\fkX)$ and $V(\fkX_s)$ along $\pi$, $d_{\fkX^{\log}_s,p}$ and the restriction of the canonical distance $d_{\fkX^{\an}_{\eta}}$ to $S(\fkX)$ therefore coincide.

\begin{prop}
Keep the assumptions and the notations of proposition (\ref{BerkRecovery}). Then $\bar\alpha$ preserves the points of type $2$ and $3$ and preserves the canonical distance on points of type $2$: if $x,y\in X_{1,[2]}$, then $d(\bar\alpha(x),\bar\alpha(y))=d(x,y)$.
\end{prop}
\begin{proof}
In $\overline X^{\an}_1$, a point $x$ is of type 2 if and only if $\overline X^{\an}_1\backslash\{x\}$ has more than two connected components and is of type 3 if and only if $\overline X^{\an}_1\backslash\{x\}$ has exactly two connected component. In particular points of type 2 (resp. 3) are preserved by arbitrary automorphisms, and therefore by $\bar \alpha$.

Let $x,y\in X^{\an}_{1,[2]}$ and let $p=(x=x_0,A_1,x_1,A_2,\cdots, A_n,x_n=y)$ be an annular path. By resolution of non singularities, there exists a finite étale cover $f:Y_1\to X_1$ such that for every $i\leq n$, $f^{-1}(x_i)\cap V(Y_1)\neq \emptyset$. One can moreover assume that $f:Y_1\to X_1$ is a Galois cover, and let $G=\Gal(Y_1/X_1)$. Then, since $V(Y_1)$ is $G$-equivariant, for every $i$, $f^{-1}(x_i)\subset V(Y_1)$. Up to refining $p$, one may assume that one can lift $p$ to a path $(y_0,B_1,y_1,B_2, \cdots,B_n, y_n)$ of $Y_1$ such that $B_i=\spec_{Y_1}^{-1}(e_i)$ where $e_i$ is a node of the special fiber $\fkY_{1,s}$ of the stable model of $Y_1$. Then $r_{\mathrm{norm}}(A_i)=l_{\fkY_{1,s}^{\log}}(e_i)/\lvert \Stab_G(\tilde e_i)\rvert$, and $l(p)=\sum_i l_{\fkY_{1,s}^{\log}}(e_i)/\lvert \Stab_G(\tilde e_i)\rvert$. Let $Y_2$ be the corresponding cover of $X_2$. Since the isomorphism of graphs $\G_{Y_1}\to \G_{Y_2}$ is $G$-equivariant and actually comes from an isomorphism of log-schemes $\fkY_{1,s}^{\log}\to\fkY_{2,s}^{\log}$, one deduces that $\bar\alpha$ preserves the length of $p$, and therefore the metric.

\end{proof}

\begin{prop}\label{TypeConservation}
Keep the assumptions and the notations of proposition (\ref{BerkRecovery}). Then $\bar\alpha$ preserves the types of the points.
\end{prop}
\begin{proof}
In $\overline X^{\an}_1$, a point $x$ is of type 2 if and only if $\overline X^{\an}_1\backslash\{x\}$ has more than two connected components and is of type 3 if and only if $\overline X^{\an}_1\backslash\{x\}$ has exactly two connected component. In particular points of type 2 (resp. 3) are preserved by arbitrary automorphisms, and therefore by $\bar \alpha$. We still need to distinguish points of type 1 from points of type 4. $(\overline X^{\an}_i)_{(2)}$ is endowed with a natural metric $d_i$ and a point $x$ of  $\overline X^{\an}_i$ is of type 1 if and only if for every sequence $(x_n)\in (\overline X^{\an}_i)_{(2)}^\NN$ that converges to $x$, $d(x_0,x_n)$ goes to infinity.
But the restriction of $d_i$ to $V(X_i)$ can be recovered (up to a multiplicative constant, which is determined by the absolute ramification index of $K_i$, which can be recovered from $\gf(X_i)$) by the log-structure on $\mathfrak X_{i,s}$. By doing so for every $\mathfrak Y_{i,s}$, one gets that for every $x,x'\in (\overline X^{\an}_1)_{(2)}$, $d_1(x,x')=d_2(\bar\alpha(x),\bar\alpha(x'))$. But a point $x$ of  $\overline X^{\an}_i$ is of type 1 if and only if for every sequence $(x_n)\in (\overline X^{\an}_i)_{(2)}^\NN$ that converges to $x$, $d(x_0,x_n)$ goes to infinity. Therefore, $\bar\alpha$ also preserves points of type 1.
\end{proof}

\begin{thm}
Let $X_1/K_1$ and $X_2/K_2$ be two p-adic hyperbolic curves satisfying rns and let $\widetilde X_i/\widetilde K_i$ be a universal cover of $X_1/K_1$. Let $\alpha:\pi_1(X_1,\widetilde X_1)\to\pi_1(X_2,\widetilde X_2)$ be an isomorphism of fundamental groups.
 Then $\alpha$ is induced by a unique isomorphism $\widetilde X_1\to\widetilde X_2$. 
\end{thm}
\begin{proof}
Let $\tilde x_1$ be a closed point of $\tilde X_1$. Then the decomposition group $D_{\tilde x_1}$ is the stabilizer of $\tilde x_1$ for the action of $\gf(X_1)$ on $\widetilde X_1|_c$. Therefore, the equivariance of $\bar\alpha$ tells us that $\alpha(D_{\tilde x_1})=D_{\bar\alpha(\tilde x_1)}$. According to Prop. \ref{BerkRecovery}, $\bar\alpha(\tilde x_1)$ is of type 1, so that $\cH(x_1)$ is a closed subfield of $\bbC_p$. Since $k(x_1)$ is a finite extension of $\Q_p$, the image of $D_{\tilde x_1}$ inside $G_{K_1}$ is open, therefore $\alpha(D_{\tilde x_1})=D_{\bar\alpha(\tilde x_1)}$ maps to an open subgroup of $G_{K_2}$, so that $\cH(x_2)$ is a finite extension of $\Q_p$. Therefore  $\alpha(D_{\tilde x_1})$ is the decomposition group of a closed point of  $\tilde X_2$. The result then follows from \cite[Thm 2.9]{AAG2}.
\end{proof}

\section{Geometric characterization of the type of points}
Let $\C$ be a algebraically closed complete non archimedean field of mixed characteristics $(0,p)$.
\begin{prop}
 Let $D=\cM(\C\{T\})$ be a disk; let $\eta$ be its Gauss point, and let $c\in H^1(D,\Z_p(1))=\Hom(\get(D),\Z_p(1))$. For $n\geq 1$, let $f_n:Y_n\to D$ be the corresponding $\mu_{p^n}$-torsor, so that $\get(Y_n)=c^{-1}(p^n\Z_p(1))\subset\get(D)$.
\begin{enumerate}
\item $\lvert\frac{HT_D(c)}{dT}\rvert_{\eta}\leq 1$;
\item For every $n\geq 1$, if $f_n$ splits over $\eta$, then 
$\lvert\frac{HT_D(c)}{dT}\rvert_{\eta}\leq \frac{1}{p^n}$.
\end{enumerate}
\end{prop}
\begin{proof}
\begin{enumerate}
\item 
If $c=\kappa(f)$ with $f\in O^\times(D)$, then $HT_D(c)=\frac{df}{f}$. Up to multiplying $f$ by a constant in $\C$, one can assume that $\lvert f\rvert_{\eta}=1$. Then $\lvert\frac{HT_D(c)}{dT}\rvert_{\eta}=\lvert\frac{df}{dT}\rvert_{\eta}\leq 1$.

In general, the restriction of $c$ to any proper closed disk is in the image of the Kummer map. In particular, if $z$ is the Gauss point of a closed disk of radius $r\in (0,1)\cap \lvert \C^\times\rvert$, one deduce from the previous argument applied to the parameter $a^{-1}T$ where $\lvert a\rvert=r$ that $\lvert\frac{HT_D(c)}{dT}\rvert_{\eta_r}\leq \frac{1}{r}$. When $r$ converges to 1, one gets that $\lvert\frac{HT_D(c)}{dT}\rvert_{\eta}\leq 1$.
\item Assume $f_n$ splits over $\eta$. Then it splits over a neighborhood of $\eta$, which contains a small annulus $\{r\leq \lvert T\rvert\leq 1\}$. By gluing with the trivial $\mu_{p^n}$-torsor over the disk  $\{0\leq \lvert \frac{1}{T}\rvert\leq \frac{1}{r}\}$ centered at infinity, one obtains a $\mu_{p^n}$-torsor over $\bbP^1$, which is necessarily trivial. Therefore $f_n$ is a trivial $\mu_{p^n}$-torsor.

The pullback $f_n^*c\in \Hom(\ga(Y_n),\Z_p(1))$ has its image in $\Hom(\ga(Y_n),p^n\Z_p(1))$, so there exists $c'\in \Hom(\ga(Y_n),\Z_p(1))$ such that $f_n^*c=p^nc'$, and so $f_n^*HT_{Y_n}(c)=p^nHT_{Y_n}(c')$. If $f_n$ is a trivial cover, let $s:D\to Y_n$ be a section of $f_n$. Then $HT_D(c)=p^nHT_D(s^*c')$ and therefore $\lvert\frac{HT_D(c)}{dT}\rvert_{\eta}=\frac{1}{p^n}\lvert\frac{HT_D(s^*c')}{dT}\rvert_{\eta}\leq \frac{1}{p^n}$.

\end{enumerate}
\end{proof}

\begin{prop}\label{NonSplit}
Let $X$ be a hyperbolic curve over $\C$. Let $x\in X^{\an}$ be a type 4 point. Let $c\in Hom(\get(X),\Z_p(1))$ such that $\HT_X(c)\neq 0$. Then $c$ is not split over $x$: its restriction $c_{D_x}\in \Hom(D_x,\Z_p(1))$ to the decomposition group $D_x$ is non-zero.

\end{prop}
\begin{proof}
Let $f_n:Y_n\to X$ be the $\mu_{p^n}$-torsor induced by $c$.

The point $x$  is contained in a closed disk. Since $\HT(c)$ vanishes in $D$ only at finitely many points of type 1, one can assume after replacing $D$ by a smaller disk that $\HT(c)$ does not vanish on $D$. Let $T$ be a parameter on $D$.
Then $\frac{HT(c)}{dT}$ has constant norm $\lambda$ on $D$. Let us normalize the radii of subdisks of $D$ such that $r(\{\lvert T-a\rvert\leq \alpha\})=\alpha$. 
Then if $f_n:Y_n\to X$ is split over a strict closed disk $D$ then $r(D)\leq \frac{\lambda}{p^n}$.
Then there exists a decreasing sequence of strict closed disks $(D_k)_{k\in\N}$ in $D$ such that $x=\lim_{k}\eta_{D_k}$ where $\eta_{D_k}$ is the Gauss point of $D_k$. Since $x$ if of type 4, $r:=\lim_k r(D_k)>0$. In particular, there exists $n\in\N$ such that $\frac{\lambda}{p^n}<r$. Therefore, for every $k\in \bbN$, $\frac{\lambda}{p^n}<r(D_k)$, so that $Y_n\to X$ does not split over $x_k$. Therefore it does not split over $x$ either. In particular $c_{D_x}\neq 0$.
\end{proof}
This criterion gives another proof of proposition \ref{TypeConservation} that does not use the recovery of the log structure on the special fiber. This allows to deduce the following purely geometric analogue of \ref{TypeConservation}.

\begin{dfn}
Let $\Pi$ be a topological group.

 One denotes by $\RNS(\Pi)$ the following property: there exists a hyperbolic curve $X$ over an algebraically closed complete non-archimedean field $\C$ of mixed characteristics $(0,p)$ which satisfies resolution of non-singularity and an isomorphism $a:\Pi\to \gtemp(X,\widetilde X)$. 

Such an isomorphism $a:\Pi\to \gtemp(X,\widetilde X)$ will be called a geometric setting of $\Pi$.
\end{dfn}
Let $\Pi$ be a topological group such that $\RNS(\Pi)$ is true. By \cite[Th. 3.39]{gaulhiac}, one has a group-theoretical graph that recovers the graph of the stable reduction in the same way as in \cite{GphAnab} over discretely valued fields. More precisely, one sets $\G(\Pi)$ to be the graph whose vertices are conjugacy classes of maximal compact subgroups of $\Pi^{\mathrm{prime-to-}p}$ and whose edges are conjugacy classes of non-trivial intersections of two distinct maximal compact subgroups of $\Pi^{\mathrm{prime-to-}p}$. Then any geometric setting $a:\Pi\to\gtemp(X)$ induces an isomorphism $\G(\Pi)\to\G(X)$.

Then, the construction of \cite[\S 3.3]{RNSMumf} (or the construction in the proof of \ref{BerkRecovery}) extends without modification to this context (the only slight difference is that in the notations of  \cite[\S 3.3]{RNSMumf}, $A_e$ is no longer isomorphic to $\Q\cap (0,1)$ in general, but to $v(\C^\times)\cap (0,1)$ where $v(\C^\times)$ is the value group of $\C$): one obtains a topological space $\widetilde X_c(\Pi)$ endowed with a continuous action of $\Pi$, and for every geometric setting $a:\Pi\to\gtemp(X,\widetilde X)$ a unique equivariant homeomorphism $\tilde a_c:\widetilde X_c(\Pi)\to\lvert \widetilde X^{\an}\rvert_c$.

Let $ \overline X(\Pi)=\widetilde X_c(\Pi)/\Pi$. If ${\tilde x}\in\widetilde X_c(\Pi)$, let $D_{\tilde x}=\Stab_{\Pi}(\tilde x)$. If $x\in \overline X(\Pi)$, let $D_x=D_{\tilde x}$, where $\tilde x$ is a preimage of $x$ in $\widetilde X(\Pi)$, so that it is only well-defined up to conjugation. If $\Pi=\gtemp(X)$, then $D_x$ coincides with the image of the outer morphism $G_{\cH(x)}\to \Pi_X$ induced by the morphism of analytic spaces $\calM(\cH(x))\to X^{\an}$, where $G_{\cH(x)}$ is the absolute Galois group of $\cH(x)$. 

Cusps of $\overline X(\Pi)$ are end-points $x$ of $\overline X(\Pi)$ (in the sense that the set of  open neighborhoods $U$ of $x$ such that $U\backslash\{x\}$ is connected are cofinal among open neighborhoods of $x$) such that $D_x$ is isomorphic to $\widehat{\Z}$, and we denote by $X(\Pi)=\overline X(\Pi)\backslash \{\mathrm{cusps}\}$. Then for every geometric setting $a:\Pi\to\gtemp(X,\widetilde X)$, $\tilde a_c$ restricts to a homeomorphism $X(\Pi)\to \lvert X^{\an}\rvert$ (it follows for example from the following proposition \ref{AnabTypeProp}).

\begin{dfn}\label{AnabTypeDef}
Let $x\in X(\Pi)$. Let $l$ be a prime number different from $p$. Then $x$ is said to be of type:
\begin{itemize}
\item 1 if $D_x=\{1\}$;
\item 2 if $D_x$ has a non abelian pro-$l$ subquotient;
\item 3 if $D_x$ has a non trivial pro-$l$ subquotient but every pro-$l$ subquotient of $D_x$ is abelian.
\item 4 if $D_x$ is a non-trivial pro-$p$ group.
\end{itemize}
\end{dfn}
\begin{prop}\label{AnabTypeProp}
For any geometric setting $a:\Pi\to \gtemp(X)$, the isomorphism 
$\bar a:X(\Pi)\simeq\lvert X^{\an}\rvert$ preserves the types of points. In particular, the type of a point in definition \ref{AnabTypeDef} is independent from $l$.
\end{prop}
\begin{proof}
From the definition \ref{AnabTypeDef}, it is clear that a point in $X(\Pi_Z)$ cannot be of two different types. It is therefore sufficient to show that if $x\in Z^{\an}$ is of a given type, $D_x\subset \Pi_Z$ satisfies the corresponding property in definition \ref{AnabTypeDef}.
\begin{enumerate}
\item Assume $x$ is of type $1$. Then $D_{x}=\{1\}$ because $\cH(x)=\C$ is algebraically closed.
\item Assume $x$ is of type $2$. By resolution of non-singularities, there exists a finite cover $f:Y\to Z$ and $y\in f^{-1}(\{x\})\cap V(Y)$. Then $D_y=D_x\cap\Pi_Y$ (for some choice of $D_x$ inside its conjugacy class), and the image of $D_y$ in $\Pi^{\mathrm{prime-to-}p}_Y$ is isomorphic to $\get(U)^{\mathrm{prime-to-}p}$ where $U$ is the smooth locus in the irreducible component of the special fiber of the stable model of $Y$ corresponding to the specialisation of $y$. Since $U$ is an hyperbolic curve over the residue field $\widetilde{\C}$ of $\C$, $D_y$ has a non abelian pro-$l$ quotient.

\item Assume $x$ is of type $3$. By resolution of non-singularities, there exists a finite Galois cover $f:Y\to Z$ and $y\in f^{-1}(\{x\})$ such that $y$ belongs to the skeleton of $Y^{\an}$. Then  the image of $D_y$ in $\Pi^{\mathrm{prime-to-}p}_Y$ is isomorphic to $\Z(1)^{\mathrm{prime-to-}p}$, and therefore has a non trivial pro-$l$ quotient.

Since $\cH(x)$ has an algebraically closed residue field, the absolute galois group $G_{\cH(x))}$ of $\cH(x)$ is equal to its inertia subgroup $I_{\cH(x))}$. Since the wild inertia subgroup $W_{\cH(x))}$ is a pro-$p$ group, every subquotient of $G_{\cH(x))}$ is isomorphic to a subquotient of $I_{\cH(x))}/W_{\cH(x))}$ and is therefore abelian (cf. \cite[\S 2.3.19]{ducroscurves}).
\item Assume $x$ is of type $4$. Then according to proposition \ref{NonSplit}, $D_x$ is non trivial. Since $D_x$ is a quotient of the absolute Galois group of $\cH(x)$, it is enough to show that its absolute galois group $G_{\cH(x))}$ is a $p$-group. Since the residue field of $\cH(x)$ is algebraically closed and its value group is divisible, $G_{\cH(x)}$ is equal to its wild inertia subgroup $W_{\cH(x)}$, which is a pro-$p$ group.

\end{enumerate}
\end{proof}

\bibliography{biblio}
\bibliographystyle{amsplain}

\end{document}